\pdfoutput=1
\documentclass{amsart}
\usepackage{amssymb}
\usepackage{graphicx}
\usepackage{enumerate}
\usepackage{multirow}
\usepackage{amsmath,color}
\usepackage{hyperref}
\usepackage{url}

\usepackage{tikz,amsaddr}

\newtheorem{theorem}{Theorem}[section]

\newtheorem{proposition}[theorem]{Proposition}
\newtheorem{prop}[theorem]{Proposition}

{

\title{2-walk-regular dihedrants from group-divisible designs}

\author{Zhi Qiao$^1$, Shao Fei  Du$^2$ and Jack H. Koolen$^{1,3,*}$}
\address{$^1$School of Mathematical Sciences, University of Science and Technology of China, 96 Jinzhai Road, Hefei, 230026, Anhui, PR China}
\address{$^2$Department of Mathematics, Capital Normal University, Beijing 100048, PR China}
\address{$^3$Wen-Tsun Wu Key Laboratory of Chinese Academy of Sciences, University of Science and Technology of China, 96 Jinzhai Road, Hefei, 230026, Anhui, PR China}
\thanks{$^*$Corresponding author}


\begin{document}
\maketitle

\begin{abstract}In this note, we construct bipartite 2-walk-regular graphs with exactly 6 distinct eigenvalues as  incidence graphs of group-divisible designs with the dual property. For many of them, we show that they are 2-arc-transitive dihedrants. We note that  many of these graphs are not described in Du et al.\cite[Theorem 1.2]{Du08},  
in which they classify the connected 2-arc transitive dihedrants. 
\end{abstract}

\hspace{2cm}

\section{Introduction}
For unexplained terminology, see next section. 
C. Dalf\'{o} et al.\cite{Dalfo11} showed the following result. 

\begin{prop}(cf. \cite[Proposition 3.4, 3.5]{Dalfo11})\label{dalfo}
Let $s,d$ be positive integers. 
Let $\Gamma$ be a connected $s$-walk-regular graph with diameter $D \geq s$ and with exactly $d+1$ distinct eigenvalues. 
Then the following hold:\\
(i) If $d \leq s+1$, then $\Gamma$ is distance-regular;\\
(ii) If $d \leq s+2$ and $\Gamma$ is bipartite, then $\Gamma$ is distance-regular.
\end{prop}

In this note, 
we will construct infinitely many bipartite 2-walk-regular graphs with exactly 6 distinct eigenvalues and diameter $D =4$, 
thus showing that Statement (ii) of Proposition \ref{dalfo} is not true for $d=5$ and $s=2$. 
We will construct these graphs as the point-block incidence graphs of certain group-divisible designs having the dual property. 
We will show that infinitely many of these graphs are 2-arc transitive dihedrants, 
and, en passant, provide a new description of 2-arc transitive graphs found by Du et al. \cite{Du98}.
Note that, although most of the graphs we describe may not be new, 
the fact that many of them are 2-arc-transitive dihedrants seems to be new  
as they give counter examples to a result of Du et al. \cite[Theorem 1.2]{Du08} 
in which they classify the connected 2-arc transitive dihedrants. The classical examples $\Gamma(d,q)$ $(d \geq 2$ and $q$ a prime power$)$, as described in Section 4, were not mentioned in  Du et al. for the case $d \geq 3$ and $q$ any prime power, and also for the case$d=2$ and $q$ a power of two.

\section{Preliminaries}


All the graphs considered in this paper are finite, undirected and simple. 
The reader is referred to \cite{BCN} for more information. 
Let $\Gamma$ be a connected graph with vertex set $V=V(\Gamma)$, edge set $E=E(\Gamma)$, and denote $x\sim y$ if the vertices $x,y\in V$ are adjacent. 
The {\em distance} $d_{\Gamma}(x,y)$ between two vertices $x,y\in V$ is the length of a shortest path connecting $x$ and $y$ in $\Gamma$. 
If the graph $\Gamma$ is clear from the context, then we simply use $d(x,y)$. 
The maximum distance between two vertices in $\Gamma$ is the {\em diameter} $D=D(\Gamma)$. 
We use $\Gamma_i(x)$ for the set of vertices at distance $i$ from $x$ and denote $k_i(x)=|\Gamma_i(x)|$. 
For the sake of simplicity, we write $\Gamma(x)=\Gamma_1(x)$ and $k(x)=k_1(x)$. 
The {\em valency} of $x$ is the number $|\Gamma(x)|$ of vertices adjacent to it. 
A graph is {\em regular} with valency $k$ if the valency of each of its vertices is $k$.

The {\em distance-$i$ matrix} $A_i=A(\Gamma_i)$ is the matrix whose rows and columns are indexed by the vertices of $\Gamma$ and the $(x,y)$-entry is $1$ whenever $d(x,y)=i$ and $0$ otherwise. 
The {\em adjacency matrix} $A$ of $\Gamma$ equals $A_1$.

A connected graph $\Gamma$ with diameter $D$ is called {\em distance-regular} 
if there are integers $b_i$, $c_i$ $(0\leq i\leq D)$ such that for any two vertices $x,y\in V(\Gamma)$ with $d(x,y)=i$, 
there are precisely $c_i$ neighbours of $y$ in $\Gamma_{i-1}(x)$ and $b_i$ neighbours of $y$ in $\Gamma_{i+1}(x)$, 
where we define $b_D=c_0=0$. 
In particular, any distance-regular graph is regular with valency $k=b_0$. 
We write $a_i := k - b_i -c_i$ where $i =1,2, \ldots, D$. 


We generalize the intersection numbers of distance-regular graphs to all graphs. 
For a graph $\Gamma$ and two verices $x,y\in V$ with $d(x,y)=i$, we define $p^i_{j h}(x,y)=|\Gamma_j(x)\cap\Gamma_h(y)|$. 
Then we define $c_i(x,y)=p^i_{i-1,1}(x,y)$, $a_i(x,y)=p^i_{i 1}(x,y)$, $b_i(x,y)=p^i_{i+1,1}(x,y)$ and we have $k(y)=c_i(x,y)+a_i(x,y)+b_i(x,y)$. 

A graph is {\em $t$-walk-regular} if the number of walks of every given length $l$ between two vertices $x,y\in V$ depends on the distance between them, provided that $d(x,y)\leq t$ (where it is implicitly assumed that the diameter of the graph is at least $t$). 

If a graph $\Gamma$ is $t$-walk-regular, then for any two vertices $x$ and $y$ at distance $i$, 
the numbers $c_i=c_i(x,y), a_i=a_i(x,y), b_i=b_i(x,y)$ are well-defined for $0 \leq i\leq t$ (see \cite[Proposition 3.15]{Dalfo11}) and the numbers $p^i_{j h}=p^i_{j h}(x,y)$ are also well-defined for $0\leq i,j,h\leq t$ (see \cite[Proposition 1]{Dalfo10}). 
And in a $t$-walk-regular graph $\Gamma$, for any vertex $x$, use the relation $k_{i-1}(x)b_{i-1}= k_{i}(x)c_{i}$, we obtain that $k_i=k_i(x)$ are well-defined for $0\leq i\leq t$. 
If the diameter $D=t+1$, then $k_{t+1}=k_{t+1}(x)$ is also well-defined.

Let $\Gamma$ be a graph. 
The {\em eigenvalues} of $\Gamma$ are the eigenvalues of its adjacency matrix $A$. 
We use $\{\theta_0>\cdots>\theta_d\}$ for the set of distinct eigenvalues of $\Gamma$. 
If $\Gamma$ has diameter $D$, then since $I, A, \ldots, A^D$ are linearly independent, it follows that $d\geq D$. 
The {\em multiplicity} of an eigenvalue $\theta$ is denoted by $m(\theta)$. 

Let $\Gamma$ be a graph. 
Let $\Pi= \{P_1, P_2, \ldots, P_t\}$ be a partition of the vertex set of $\Gamma$ where $t$ is a positive integer. 
We say $\Pi$ is an equitable partition if there exists non-negative integers $q_{ij}$ $1 \leq i, j \leq t$ such that any vertex in $P_i$has exactly $q_{ij}$ neighbours in $P_j$. 
The $(t \times t)$-matrix $Q = (q_{ij})_{1 \leq i,j \leq t}$ is called the quotient matrix of $\Pi$. 
If $\Pi$ is equitable, the distribution diagram with respect to $\Pi$ is the diagram in which we present each $P_i$ by a balloon such that the balloon representing $P_i$ is joined by a line segment to the balloon representing $P_j$ if $q_{ij} >0$ and we will write the number $q_{ij}$ just 
above the line segment close to the balloon representing $P_i$. 
Inside the balloon representing $P_i$, we write $p_i:= \# P_i$. 
For example the distance-partition of of a vertex $x$ of a distance-regular graph with diameter $D$ has distribution diagram as in FIGURE \ref{fg:1}
\\ \\ \\ 
%

\begin{figure}[h!]
	\scalebox{.75}{
		\includegraphics{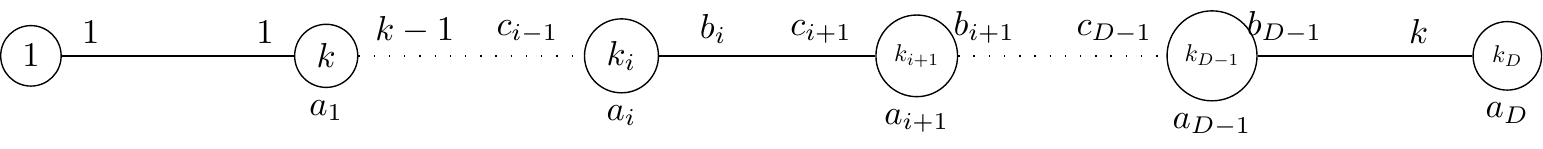}
	}
	\caption{}
	\label{fg:1}
\end{figure}

Let $G$ be a group with identity $1$ and let $Q$ be a subset of $G^*:= G-\{1\}$ closed under taking inverses. 
Then the {\em Cayley graph} Cay$(G,Q)$ is the undirected graph with 
vertex set $G$ and edge set $E($Cay$(G,Q))=\{ \{g,h\} \mid g^{-1}h\in Q\}$. 
It is known that Cay$(G, Q)$ is vertex-transitive and it is connected if and only if $Q$ generates $G$. 

The {\em dihedral group} of order $2n$ is the group $D_{2n}=\langle a,b\mid a^n=1=b^2, 
bab=a^{-1}\}$. 
Let $Q$ be a subset of $D^*_{2n}$ closed under taking inverses. 
The graph $Cay(D_{2n},Q)$ is called a dihedrant and is denoted by Dih$(2n,S,T)$ where 
$Q=\{a^i\mid i\in S\}\cup \{a^j b\mid j\in T\}$.

Let $G$ be a finite group of order $mn$ and $N$ a normal subgroup of $G$ of order $n$. 
A $k$-element subset $D$ of $G$ is called an {\em $(m, n; k, \lambda)$-relative difference set} in $G$ relative to $N$ 
if every element in $G\setminus N$ has 
exactly $\lambda$ representations $r_1 r_2^{-1}$ (or $r_1-r_2$ if $G$ is additive) with $r_1, r_2\in D$, 
and no non-identity element in $N$ has such a representation. 
When $n=1$, $D$ is a {\em $(m,k,\lambda)$-difference set}. 
A difference set or relative difference set is called {\em cyclic} if the group $G$ is cyclic. 
Note that any cyclic relative difference set or difference set can be seen as a relative difference set of $Z_{mn}$ or difference set of $Z_m$, respectively. 

A {\em incident structure} consists of a set $\mathcal P$ of points, a set $\mathcal B$ of blocks (disjoint from $\mathcal P$), 
and a relation $I \subseteq  \mathcal P\times \mathcal B$ called {\em incidence}. 
If $(p,B)\in I$, then we say the point $p$ and the block $B$ are {\em incident}. 
We usually will consider the blocks $B$ as subsets of $\mathcal P$.
If $\mathcal I=\{\mathcal P, \mathcal B, I\}$ is an incidence structure, then its {\em dual} incident structure is given by 
$\mathcal I^*=\{ \mathcal B,\mathcal P, I^*\}$, where $I^*=\{(B,p) \mid (p,B)\in I\}$. 
The {\em incident graph} of $\Gamma(\mathcal I)$ of an incident structure $\mathcal I$ is 
the graph with vertex set $\mathcal P\cup \mathcal L$, where two vertices are adjacent if and only if they are incident. 

A {\em group divisible design} $GDD(n,m;k;\lambda_1,\lambda_2)$ is an ordered triple $(\mathcal P,\mathcal G,\mathcal B)$ 
where $\mathcal P$ is a set of points, 
$\mathcal G$ is a partition of $\mathcal  P$ into $m$ sets of size $n$, each set being called a {\em group}, 
and $\mathcal B$ is a collection of subset of $\mathcal P$, called {\em blocks}, each of size $k$, 
such that 
\begin{enumerate}[i)]
	\item each pair of points that occur together in the same group occur together in exactly $\lambda_1$ blocks, 
	\item each pair of points that occur together in no group occur together in exactly $\lambda_2$ blocks. 
\end{enumerate}

The {\em dual} of a group divisible design is the corresponding dual incident structure. 
A group divisible design $GDD(n,m;k;\lambda_1,\lambda_2)$ has the {\em dual property} 
if its dual is a group divisible design with the same parameters, and the it is denoted by $GDDDP(n,m;k;\lambda_1,\lambda_2)$. 

Let $X$ be a finite set and $\mathbb C^{X\times X}$ the set of complex matrices with rows and columns indexed by $X$. 
Let $\mathcal R=\{R_0, R_1, \ldots , R_n\}$ be a set of non-empty subsets of $X\times X$. 
For each $i$, let $A_i\in \mathbb C^{X\times X}$ be the adjacency matrix of the (in general, directed) graph $\Gamma^{\mathcal R}_i:=(X,R_i)$. 
We call $\Gamma^{\mathcal R}_i$ the {\em relation graph} of $\mathcal R$ with respect to relation $R_i$. 
The pair $(X,\mathcal R)$ is an {\em association scheme} with $n$ {\em classes} if 
\begin{enumerate}[i)]
	\item $A_0=I$, the identity matrix, 
	\item $\sum_{i=0}^n A_i=J$, the all ones matrix, 
	\item $A_i^t\in\{A_0,A_1,\ldots, A_n\}$ for $0\leq i \leq n$, 
	\item $A_i A_j$ is a linear combination of $A_0, A_1, \ldots, A_n$ for $0\leq i, j\leq n$. 
\end{enumerate}
The vector space $\bf A$ spanned by $A_i$ is the {\em Bose-Mesner algebra} of $(X,\mathcal R)$. 

We say that $(X,\mathcal R)$ is {\em commutative} if $\bf A$ is commutative, 
and that $(X,\mathcal R)$ is {\em symmetric} if the $A_i$ are symmetric matrices. 
A symmetric association scheme is commutative. 
The distance distribution diagram of the relation graph of the relation $R_1$ of a symmetric association scheme $(X,\mathcal R)$ is the distribution diagram with respect to the equitable partition $\Pi = \{P_0, P_1, \ldots, P_n\}$ of $\Gamma^{\mathcal R}_1$, where $P_i = \{ y \in X \mid (x,y) \in R_i\}$ for $i=0,1, \ldots, n$ and $x$ a fixed vertex of $X$.

\section{Group-divisible designs having the dual property}
In this section we will construct bipartite 2-walk-regular graphs with diameter 4 having exactly 6 distinct eigenvalues, 
as the point-block incidence graph of certain group-divisible designs having the dual property. 

\begin{theorem}\label{gdddp}
Let ${\mathcal D}$ be a (non-empty) GDDDP$(n, m; k; 0 ,\lambda_2)$ with $n, m \geq 2$ 
and let $\Gamma := \Gamma({\mathcal D})$ be the point-block incidence graph of $\mathcal D$. 
Let $A$ be the adjacency matrix of $\Gamma$. 
Then $\Gamma$ is a relation graph, say with respect to (connecting) relation $R$, of a symmetric association scheme $\mathcal X$ with 5 classes, such that the distribution diagram of $\mathcal X$ with respect to $R$ is as in FIGURE \ref{fg:2}, where $k_4=n-1$, $c_2=\lambda_2$ and $b_2'=k_4c_2$.  
In particular, $\Gamma$ is a bipartite 2-walk-regular graph with diameter 4 and exactly 6 distinct eigenvalues. 

%

\begin{figure}[h!]
	\includegraphics{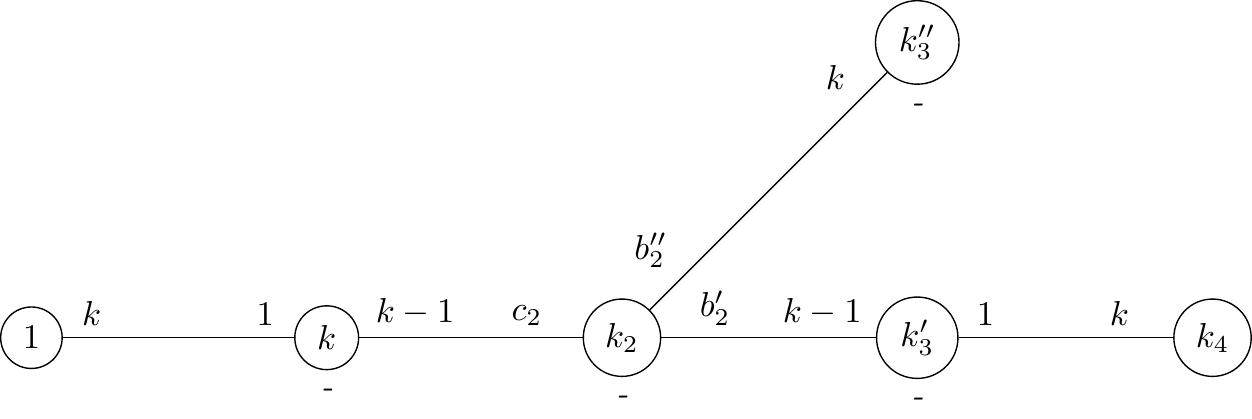}
	\caption{}
	\label{fg:2}
\end{figure}

\end{theorem}
\begin{proof}
{\bf the proof is modified}

Note that $\Gamma$ is bipartite. 
Let $x$ be a vertex of $\Gamma$. 
By the dual property, we may assume without loss of generality that $x$ is a point of $\mathcal D$, thus $\mathcal P=\{x\}\cup\Gamma_2(x)\cup\Gamma_4(x)$. 
As $\lambda_2\neq 0$ holds, we find that the group of $\mathcal D$ that contains the vertex $x$ is $\{x\}\cup \Gamma_4(x)$.  Also $c_2(x,y)=\lambda_2$ holds for any vertex $y\in\Gamma_2(x)$. 
Now let $\Gamma_3'(x)$ be the vertices at distance 3 from $x$ with a neighbour in $\Gamma_4(x)$ 
and $\Gamma_3''(x) := \Gamma_3(x) \setminus \Gamma'_3(x)$. 
As $\lambda_1=0$, we see that $b_3(x,y)=1$ for any vertex $y\in\Gamma_3'(x)$. 
It follows that the partition $\Pi = \{ \{x\}, \Gamma(x), \Gamma_2(x), \Gamma_3'(x), \Gamma_3''(x), \Gamma_4(x)\}$ is an equitable partition of $\Gamma$ with distribution diagram as in FIGURE \ref{fg:2} with $k_4=n-1$, $c_2=\lambda_2$ and $b_2'=k_4c_2$. 
Now define the matrix $B_3$ by $(B_3)_{x y} = 1$ if $y \in \Gamma_3'(x)$ and 0 otherwise, where $x$ and $y$ are any two vertices of $\Gamma$. 
Note that  for any pair of vertices  $x$ and $y$ with  $y\in\Gamma_3(x)$, the number of walks of length 3 between $x$ and $y$ equals $c_3(x,y)c_2=c_3(y,x)c_2$, and $c_3(x,y)\neq k$ means $y\in \Gamma_3' (x)$ , which in turn  implies that $B_3$ is symmetric. 
Let $C_3 = A_3 -B_3$, where $A_i$ is the distance-$i$ matrix for $\Gamma$ for $i=0,1,2,3,4$. 
It is straightforward to check that the set of matrices $\{A_0= I, A_1, A_2, B_3, C_3, A_4\}$ satisfies the axioms of a symmetric association scheme. 
That $\Gamma$ is 2-walk-regular follows from the fact that $A_2$ is a relation matrix of the association scheme. 
As $\Gamma$ is the relation graph of a 5-class association scheme, it follows that $\Gamma$ has at most 6 distinct eigenvalues. 
The fact that it has at least 6 eigenvalues follows for Proposition \ref{dalfo}. 
This shows the theorem. 
\end{proof}
{\bf Remark.}\\
(i) The first and second eigenmatrices of the corresponding association schemes, where $b_2''=k-(k_4+1)c_2$ are as follows:

$$
P=\begin{pmatrix}
 1 & k & \frac{(k-1) k}{c_2} & k k_4 & \frac{(k-1) b_2''}{c_2} & k_4 \\
 1 & -k & \frac{(k-1) k}{c_2} & -k k_4 & -\frac{(k-1) b_2''}{c_2} & k_4 \\
 1 & \sqrt{k} & 0 & -\sqrt{k} & 0 & -1 \\
 1 & -\sqrt{k} & 0 & \sqrt{k} & 0 & -1 \\
 1 & \sqrt{b_2''} & -k_4-1 & k_4 \sqrt{b_2''} & -(k_4+1) \sqrt{b_2''} & k_4 \\
 1 & -\sqrt{b_2''} & -k_4-1 & -k_4 \sqrt{b_2''} & (k_4+1) \sqrt{b_2''} & c_2 
\end{pmatrix},
$$
$$
Q=\begin{pmatrix}
 1 & 1 & \frac{(k^2-b_2'') k_4}{ k-b_2''} & -\frac{(k^2-b_2'') k_4}{ b_2''-k} & \frac{(k-1) k}{k- b_2''} & \frac{(k-1) k}{k- b_2''} \\
 1 & -1 & \frac{(k^2- b_2'') k_4}{\sqrt{k} (k- b_2'')} & -\frac{ (k^2-b_2'') k_4}{\sqrt{k} (k- b_2'')} & \frac{\sqrt{ b_2''} (k-1)}{k- b_2''} & -\frac{\sqrt{ b_2''} (k-1)}{k- b_2''} \\
 1 & 1 & 0 & 0 & -1 & -1 \\
 1 & -1 & -\frac{ k^2-b_2''}{\sqrt{k} (k- b_2'')} & \frac{k^2- b_2''}{\sqrt{k} (k- b_2'')} & \frac{\sqrt{ b_2''} (k-1)}{k- b_2''} & -\frac{\sqrt{ b_2''} (k-1)}{k- b_2''} \\
 1 & -1 & 0 & 0 & -\frac{k}{\sqrt{ b_2''}} & \frac{k}{\sqrt{ b_2''}} \\
 1 & 1 & -\frac{k^2- b_2''}{ k-b_2''} & -\frac{k^2- b_2''}{ k-b_2''} & \frac{(k-1) k}{k- b_2''} & \frac{(k-1) k}{k- b_2''}
\end{pmatrix}.
$$
\noindent
(ii)  It is easy to see that if $\mathcal X$ is a symmetric association scheme 
such that for a relation $R$ the distribution diagram of $\mathcal X$ with respect to $R$ is equal to the diagram in FIGURE \ref{fg:2}, then $\mathcal X$ comes from a GDDDP$(n, m;k; 0, \lambda_2)$ with $m=k_4+1$, $\lambda_2=c_2$, as described in the 
above theorem. \\
(iii) If the point-block incidence matrix of a GDDDP is symmetric with zeroes on the diagonal, 
they correspond exactly with the divisible design graphs as defined by Haemers et al. \cite{Haemers11}. 
They correspond to exactly the GDDDP have a polarity without absolute points.

\section{Classical Examples}
In this section we discuss classical examples of group divisible designs with the dual property 
and show that the point-block incidence graphs of these examples are 2-arc transitive dihedrants.
For more information see  \cite{Bose44,Elliot68}.

Let $d \geq 2$ be an integer and let $q$ be a prime power. 
Let $V$ be a vector space of dimension $d$ over the finite field with $q$ elements, GF$(q)$. 
Let $X$  be the set of non-zero elements of $V$. 
For $x \in X$, let $G_x = \{ \alpha x \mid \alpha \in$ GF$^*(q) := $GF$(q) \setminus \{0\} \}$ and ${\mathcal G }:= \{ G_x \mid x \in X\}$. 
Let ${\mathcal B} := \{ x + H \mid  x \in X, H$ a hyperplane in $V, x \not \in H\}$, 
where $x + H = \{x + h \mid h \in H\}$.
Then the design ${\mathcal D}(d, q) := (X, {\mathcal G}, {\mathcal B})$ is a GDDDP$(q-1, \frac{q^d-1}{q-1}; q^{d-1}; 0, q^{d-2})$. 
It is clear that the general linear group GL$(d,q)$ acts as a group of automorphisms of ${\mathcal D}(d, q)$ 
such that its subgroup $Z:= \{ \alpha I_d \mid \alpha  \in $ GF$^*(q) \}$fixes $G_x$ for all $x \in X$.

\begin{proposition}
For all integers $d \geq 2$ and prime powers $q$, the incidence graph $\Gamma(d,q)$ of $\mathcal D(d,q)$ is a 2-arc transitive dihedrant. 
\end{proposition}  

\begin{proof}

Let $z$ be a primitive element of GF$^*(q^d)$ and define the map $\tau_z:X' \rightarrow X'$ by $\tau_z(x) = zx$ for $x \in X'=$ GF$^*(q^d)$. The map $\tau_z$ has order $n :=q^d-1$.
We can identify the map $\tau_z$ as a linear map $A_z \in $ GL$(d,q)$, by identifying the field GF$(q^d)$ with the vector space GF$(q)^d = X \cup \{0\}$. Note that the group $\langle A_z\rangle $ is the well-known Singer--Zyklus subgroup of  GL$(d,q)$. 
It is clear that $A_z$ maps non-zero vector of $X = $GF$(q)^d$ to non-zero vectors, and  affine hyperplanes to an affine hyperplanes.  Note that $A_z$ is an automorphism 
of the graph $\Gamma(d,q)$. Let $H_y := \{ u \in  X \mid y^Tu = 1\}$. Note that $A_z$ maps, for any non-zero vector $y$,  the affine hyperplane $H_y$ to $H_{y'}$ where $y' = (A_z^T)^{-1}y$. Now let $u_0, u_1 , \ldots ,u_{n-1}$ and $v_0, v_1, \ldots ,v_{n-1}$ be two  orderings of the non-zero vectors of GF$(q)^d$ such that $A_z$ maps $u_i$ to $u_{i+1}$ and $H_{v_i}$ to 
$H_{v_{i+1}}$ (where we take the indices modulo $n$). Now define the map $\phi: X \cup {\mathcal B} \rightarrow X \cup {\mathcal B}$ by $\phi(u_i) = H_{v_{-i}}$ and
$\phi(H_{v_i}) = u_{-i}$ for $i =0,1,2, \ldots, n-1$ (where we take indices modulo $n$). Then, it is easy to see that $\phi$ defines a polarity of the design ${\mathcal D}(d, q)$
and hence an automorphism of the graph $\Gamma(d, q)$. 
Moreover the group generated by $\phi$ and $A_z$ is the dihedral group $D_{2n}$ and acts regularly on the vertex set. 
This shows that $\Gamma(d, q)$ is a dihedrant by \cite[Lemma 3.7.2]{AGT}.  

%

This also means that the graph $\Gamma(d,q)$ is vertex-transitive. 
Let $x$ be an element of $X$. 
Then $y$ is at distance 2 from $x$ in $\Gamma(d,q)$ if and only if $y \in X$ and $x,y$ are linearly independent. 
If $x,y$ and $x',y'$ are two pairs of linear independent vectors in $X$, $H$ is an affine hyperplane containing $x$ and $y$ and $H'$ is an affine hyperplane containing $x'$ and $y'$,
then there exists an element $\sigma$ of GL$(d,q)$ that maps simultaneously $x$ to $x'$, $y$ to $y'$ and $H$ to $H'$. 
This shows that $\Gamma(d,q)$ is a 2-arc transitive dihedrant.
\end{proof}

Now we will consider quotients of $\Gamma(d,q)$, which are also 2-arc transitive dihedrants. 
As $Z$ is a cyclic group of order $q-1$ and let $n  \geq 2$ be a divisor of $q-1$.  
Then $Z$ contains a cyclic subgroup $C= C_{(q-1)/n}$ of order $(q-1)/n$. 
Now consider the set of orbits 
$\mathcal O$ of $\Gamma(d,q)$ under $C$. Let $\Gamma(d, q, n)$ be the graph with 
vertex set $\mathcal O$ and  $O, O' \in {\mathcal O}$ are adjacent if there exists an edge in 
$\Gamma(d, q)$ connecting an element of $O$ with and element of $O'$. Then it is easy to 
verify that  $\Gamma(d, q, n)$ is a 2-arc transitive dihedrant such that the distance-distribution diagram with respect to any vertex is as in FIGURE \ref{fg:2}, where $k=q^{d-1}$, $c_2=\lambda=q^{d-2}(q-1)/n$, $k_4=n-1$, $b_2' = k_4c_2 $. 

\section{Cyclic relative difference sets}

In this section we give another view point on the examples of the last section 
and give a construction for dihedrants from cyclic difference sets. 

\begin{proposition}
Let $D$ be a cyclic relative difference set with parameters $(m,n;k, \lambda)$ with $m,n \geq 2$. 
Define the dihedrant  Dih$(2nm,  \emptyset, D)$.
Then Dih$(2nm, \emptyset, D)$ is the point-block incidence graph of a GDDDP$(m,n; k; 0, \lambda)$. 
In particular, Dih$(2nm, \emptyset, D)$ is a connected 2-walk-regular graph. 
\end{proposition}
\begin{proof} 
By definition the dihedrant is bipartite. 
By direct verification one sees  that the distance distribution diagram with respect to any vertex is FIGURE \ref{fg:2} with $k_4=n$ and $c_2=\lambda$. 
It easily follows that it is the point-block incidence graph of a GDDDP$(n, m; k; 0, \lambda)$. 
That the dihedrant is 2-walk-regular follows from Theorem \ref{gdddp}.
\end{proof}

The graphs $\Gamma(d, q, n)$ as considered in the last section arise from 
cyclic relative difference sets with parameters $(\frac{q^d-1}{q-1},n,q^{d-1},\frac{q^{d-2}q-1}{n})$. 
And Arasu et al. \cite{Arasu01,Arasu95} gave constructions for cyclic relative difference sets 
with parameters $(\frac{q^d-1}{q-1}, n,q^{d-1},\frac{q^{d-2}(q-1)}{n})$ for $q$ a prime power, where $n$ is a divisor of $q-1$ when q is odd or d is even, and $n$ is a divisor of $2(q-1)$ when $q$ is even and $d$ is odd. 
Arasu et al. \cite[Theorem 1.2]{Arasu01} showed that for a prime power $q$, 
cyclic relative difference sets with parameters $(\frac{q^d-1}{q-1}, n,q^{d-1},\frac{q^{d-2}(q-1)}{n})$ exists 
if and only if the above restrictions are satisfied.

There exists a 2-arc transitive dihedrant on 28 vertices with valency 4 such that the distance distribution diagram with respect to any vertex is as in FIGURE \ref{fg:3}. 
This graph corresponds to the $(7,2;4,1)$-cyclic relative difference set  $\{0,1,9,11\}$ in $Z_{14}$ relative to $\{0,7\}$. 

%

\begin{figure}[h!]
	\includegraphics{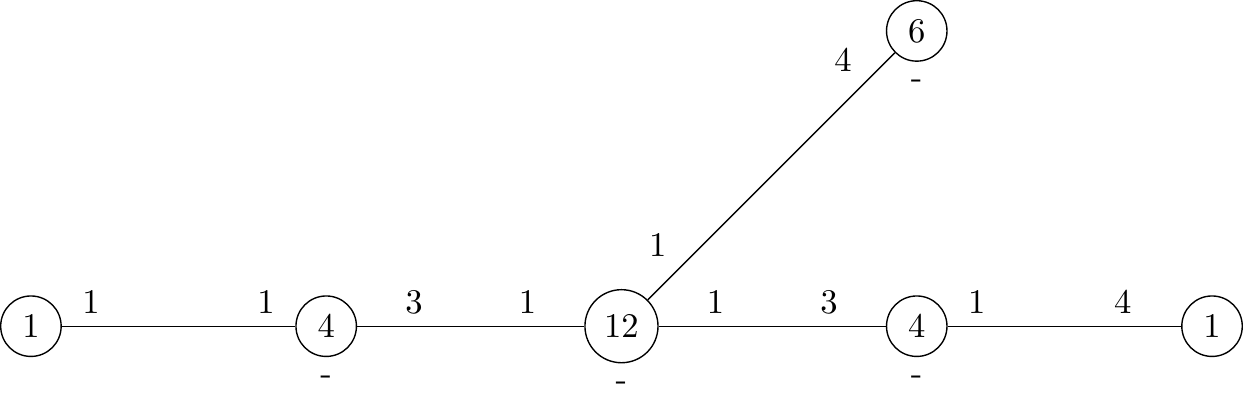}
	\caption{}
	\label{fg:3}
\end{figure}

Note that there exist cyclic relative difference sets such that the corresponding dihedrants are
not even 1-arc-transitive.
For example, the $(13,2;9,3)$-cyclic relative difference set $\{0, 9, 11, 15, 18, 19, 20, 23, 25\}$ in $Z_{26}$ relative to $\{0,13\}$ generates a dihedrant which is not 1-arc-transitive and has the same distance distribution diagram of $\Gamma(3,3)$ as in FIGURE \ref{fg:4}, which is $2$-arc transitive. 

%

\begin{figure}[h!]
	\includegraphics{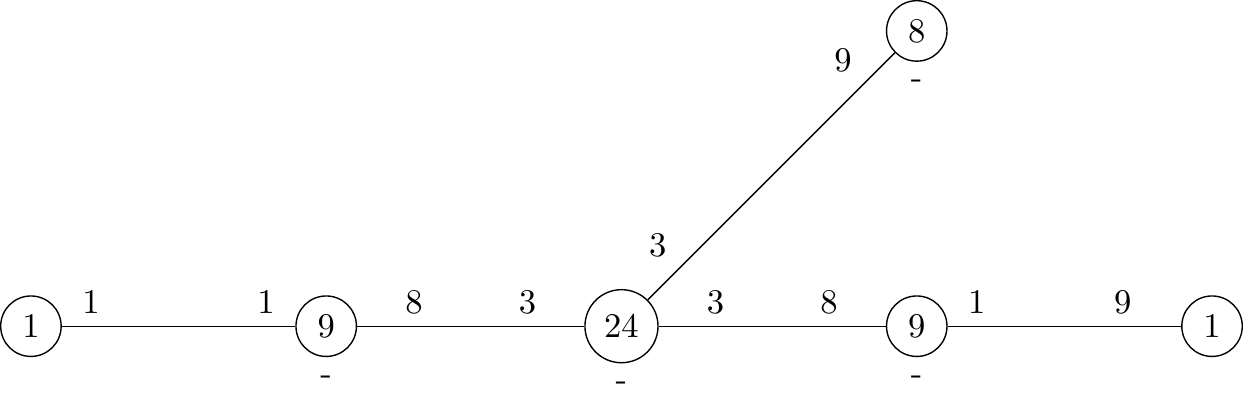}
	\caption{}
	\label{fg:4}
\end{figure}

\noindent

{\bf Remark}  The graphs $\Gamma(d, q)$ can also be described in a pure group theoretical way as a bi-coset graph (see Du and Xu \cite{Du00}). 

Take $G=$GL$(d,q)$,  and let  $R$ be the set of matrices in $G$   whose first row equals $(1, 0, 0, \ldots, 0)$; and  let $L$ be the set of matrices
in $G$ whose first column equals $( 1, 0, 0, \ldots, 0)$. Note that $R$ and $L$ are subgroups of $G$.  Then $\Gamma(d, q)$ is isomorphic to the bi-coset graph
$X=X(G, R, L, LR)$, which is bipartite with color classes  $\{ Rg \mid g \in G\}$ and $\{ Lg \mid g \in G\}$,  
where $Rg_1$ is adjacent to $Lg_2$ if and only if $g_2g_1^{-1}\in LR$.

\noindent
{\bf Acknowledgments} SFD is  partially supported by the National Natural Science Foundation of China (No.11271267) and the  National Research Foundation for the Doctoral Program of Higher Education of China (20121108110005).  JHK is partially supported by the National Natural Science Foundation of China (No. 11471009). He also acknowledges the financial support of 
the Chinese Academy of Sciences under its '100 talent' program.

\bibliographystyle{plain}
\bibliography{ref}

\end{document}